\documentclass[12pt]{amsart}
\usepackage{a4wide}
\usepackage[T1]{fontenc}
\usepackage{amssymb,amsmath,amsthm,latexsym}
\usepackage{mathrsfs}
\usepackage[usenames,dvipsnames]{color}
\usepackage{euscript}
\usepackage{graphicx}
\usepackage{mdwlist}
\usepackage{mathtools,dsfont,wasysym}
\usepackage{stmaryrd}
\usepackage[dvipsnames]{xcolor}

\DeclareFontFamily{U}{matha}{\hyphenchar\font45}
\DeclareFontShape{U}{matha}{m}{n}{
      <5> <6> <7> <8> <9> <10> gen * matha
      <10.95> matha10 <12> <14.4> <17.28> <20.74> <24.88> matha12
      }{}
\DeclareSymbolFont{matha}{U}{matha}{m}{n}
\DeclareMathSymbol{\Lt}{3}{matha}{"CE}

\DeclareFontFamily{U}{mathx}{\hyphenchar\font45}
\DeclareFontShape{U}{mathx}{m}{n}{
      <5> <6> <7> <8> <9> <10>
      <10.95> <12> <14.4> <17.28> <20.74> <24.88>
      mathx10
      }{}

\newtheorem{theorem}{Theorem}[section]
\newtheorem*{theoremA}{Theorem A}
\newtheorem*{theoremB}{Theorem B}
\newtheorem*{theoremC}{Theorem C}
\newtheorem{lemma}[theorem]{Lemma}
\newtheorem{corollary}[theorem]{Corollary}

\newtheorem{fact}[theorem]{Fact}
\newtheorem{claim}[theorem]{Claim}
\theoremstyle{remark}
\newtheorem{remark}[theorem]{Remark}
\theoremstyle{definition}
\newtheorem{definition}[theorem]{Definition}
\newtheorem{problem}[theorem]{Problem}

\numberwithin{equation}{section}
\makeatother

\newcommand{\nn}[1]{{\left\vert\kern-0.25ex\left\vert\kern-0.25ex\left\vert #1
\right\vert\kern-0.25ex\right\vert\kern-0.25ex\right\vert}}

\renewcommand{\leq}{\leqslant}
\renewcommand{\geq}{\geqslant}

\newcounter{smallromans}

\newenvironment{romanenumerate}
{\begin{list}{{\normalfont\textrm{(\roman{smallromans})}}}%
  {\usecounter{smallromans}\setlength{\itemindent}{0cm}%
   \setlength{\leftmargin}{5.5ex}\setlength{\labelwidth}{5.5ex}%
   \setlength{\topsep}{.5ex}\setlength{\partopsep}{.5ex}%
   \setlength{\itemsep}{0.1ex}}}%
{\end{list}}

\newcommand{\R}{\mathbb{R}}
\newcommand{\N}{\mathbb{N}}

\newcommand{\e}{\varepsilon}
\newcommand{\p}{\varphi}

\newcommand{\flag}{\mathord{\upharpoonright}}

\newcounter{smallromansdash}

{\end{list}}

\newcounter{bigromans} 
{\end{list}}

\begin{document}
\title[On densely isomorphic normed spaces]{On densely isomorphic normed spaces}

\author[P.~H\'ajek]{Petr H\'ajek}
\address[P.~H\'ajek]{Department of Mathematics\\Faculty of Electrical Engineering\\Czech Technical University in Prague\\Technick\'a 2, 166 27 Praha 6\\ Czech Republic}
\email{hajek@math.cas.cz}

\author[T.~Russo]{Tommaso Russo}
\address[T.~Russo]{Department of Mathematics\\Faculty of Electrical Engineering\\Czech Technical University in Prague\\Technick\'a 2, 166 27 Praha 6\\ Czech Republic}
\email{russotom@fel.cvut.cz}

\thanks{Research of the first-named author was supported in part by OPVVV CAAS CZ.02.1.01/0.0/0.0/16$\_$019/0000778.
Research of the second-named author was supported by the project International Mobility of Researchers in CTU CZ.02.2.69/0.0/0.0/16$\_$027/0008465 and by Gruppo Nazionale per l'Analisi Matematica, la Probabilit\`a e le loro Applicazioni (GNAMPA) of Istituto Nazionale di Alta Matematica (INdAM), Italy.}

\keywords{Normed spaces, dense subspaces, densely isomorphic normed spaces, biorthogonal systems in normed spaces}
\subjclass[2010]{46B20, 46B26 (primary), and 46A35, 46C05 (secondary)}
\date{\today}

\begin{abstract} In the first part of our note we prove that every Weakly Lindel\"of Determined (WLD) (in particular, every reflexive) non-separable Banach $X$ space contains two dense linear subspaces $Y$ and $Z$ that are not densely isomorphic. This means that there are no further dense linear subspaces $Y_0$ and $Z_0$ of $Y$ and $Z$ which are linearly isomorphic. 

Our main result (Theorem B) concerns the existence of biorthogonal systems in normed spaces. In particular, we prove under the Continuum Hypothesis (CH) that there exists a dense linear subspace of $\ell_2(\omega_1)$  (or more generally every WLD space of density $\omega_1$) which contains no uncountable biorthogonal system. This result lies between two fundamental results concerning biorthogonal systems, namely the construction of Kunen (under CH) of a non-separable Banach space which contains no uncountable biorthogonal system, and the construction of Todor\u{c}evi\'c (under Martin Maximum) of an uncountable biorthogonal system in every non-separable Banach space. 
\end{abstract}
\maketitle

\section{Introduction}
Although most research in Banach space theory focuses on the study of complete normed spaces, incomplete normed spaces play an important r\^{o}le in several instances in Functional Analysis, most notably in the study of spaces of continuous, $C^1$-smooth, or compactly supported, functions on $\R^n$, with integral norm. Just to name some directions of research, let us mention: inner product spaces \cite{Gudder, Gudder2}, norm attainment \cite{James, Phelps, BP}, supporting functionals \cite{Fonf1, Fonf2}, tilings \cite{FPZ}, smoothness \cite{Vanderwerff, hajek-locally,Johanis}, diffeomorphisms \cite{Bessaga, Dobrowolski, Azagra}.\smallskip

Our aim in this paper will be to study normed spaces with a prescribed completion, or, equivalently, to study dense subspaces of a given Banach space. We will be interested in understanding how much some structural assumptions on a Banach space influence its dense subspaces. Moreover, given a certain property of some dense subspace of a Banach space, we wish to know if such property is inherited by other dense subspaces of the same Banach space. From this perspective, the following heuristic question arises naturally.
\begin{center}\emph{How different can two dense subspaces of a Banach space be?}
\end{center}

A standard perturbation argument shows that in every separable Banach space there is a canonical (smallest) dense subspace, that is densely contained in every other dense subspace. More precisely, if $\{e_j;e^*_j\}_{j=1}^\infty$ is an M-basis for a separable Banach space $X$, every dense subspace of $X$ contains a dense subspace isomorphic to ${\rm span}\{e_j\} _{j=1}^\infty$ (see Theorem \ref{Th: densely for separable}). Moreover, a similar result is also valid in $\ell_1(\Gamma)$, for every set $\Gamma$ (Theorem \ref{Th: densely for ell1}). In our note we show that this feature breaks down completely in many non-separable Banach spaces, notably even in Hilbert spaces.\smallskip

In fact, our original interest in this topic comes from the problem of the existence of smooth norms on dense subspaces (compare \cite{Vanderwerff, hajek-locally} and \cite[Problems 148, 149]{GMZ}). It was shown in \cite{hajek-locally} that every separable Banach space admits a dense subspace with $C^\infty$-smooth renorming---in sharp contrast with the space itself. In our joint paper with Sheldon Dantas \cite{DHR} we prove similar results for certain non-separable Banach spaces, \emph{e.g.}, those having an unconditional basis. However, our proofs rely on the structural properties of the selected subspaces, and do not work for a general dense subspace. In view of our main results below, it is in fact conceivable that while, \emph{e.g.}, $\ell_\infty$ has a dense subspace $X$ admitting an analytic norm, it may perhaps have another dense subspace $Y$ whose every further dense subspace fails to have a smooth renorming. The same situation in fact occurs with other properties, such as the existence of an uncountable biorthogonal system, as we show below. Our paper can be viewed as the first step towards a systematic study of dense subspaces of a given (non-separable) Banach or normed space with respect to various structural properties.\smallskip

To this end, we introduce the following definition, that will be the cardinal notion of the paper.
\begin{definition} Two normed spaces $X$ and $Y$ are said to be \emph{densely isomorphic} if there exist dense subspaces $X_0$ of $X$ and $Y_0$ of $Y$ such that $X_0$ and $Y_0$ are linearly isomorphic.
\end{definition}

The obvious fact that linear isomorphisms between normed spaces extend to isomorphisms between their completions implies that densely isomorphic normed spaces have isomorphic completions. In particular, two Banach spaces are densely isomorphic precisely when they are isomorphic. This naturally leads us to the following problem, a formal restatement of the heuristic question above.
\begin{problem}\label{Problem} Let $Y$ and $Z$ be dense subspaces of a Banach space $X$. Must $Y$ and $Z$ be densely isomorphic?
\end{problem}

Theorem \ref{Th: densely for separable} and Theorem \ref{Th: densely for ell1} below yield a positive answer for separable Banach spaces and for the spaces $\ell_1(\Gamma)$. On the other hand, our first main result shows that for a substantial class of non-separable Banach spaces the situation is different:
\begin{theoremA} Every non-separable WLD Banach space contains two dense subspaces that are not densely isomorphic.
\end{theoremA}
 A striking particular case of the result is that every non-separable Hilbert space contains two dense subspaces that are not densely isomorphic. The proof of Theorem A will be given in Section \ref{Sec: Th A} (while the definition of WLD Banach space will be briefly recalled in Section \ref{Sec: WLD}). In the particular case when the Banach space $X$ under consideration satisfies $\omega_1\leq{\rm dens}\,X\leq\mathfrak{c}$, we actually prove, in Theorem \ref{th: WLD small}, the following stronger result: there exist two dense subspaces $Y$ and $Z$ of $X$ whose every non-separable subspaces are non-isomorphic. Once more, in the case of the Hilbert space $\ell_2(\omega_1)$, this shows that, although the structure of closed subspaces of Hilbert spaces is extremely rigid, its dense subspaces can be quite diverse.\smallskip

Subsequently, we present an alternative approach to Theorem A, by proving the following more precise result, under the assumption of the Continuum Hypothesis.
\begin{theoremB}[CH] Every WLD Banach space of density character $\omega_1$ contains a dense subspace with no uncountable biorthogonal system.
\end{theoremB}

The proof of Theorem B will be given in Section \ref{Sec: Th B}, where we will also discuss why the result offers an alternative approach to Theorem A in the case of WLD Banach spaces with density equal to $\omega_1$. We borrowed the very idea of the proof from \cite[Theorem 5.1]{HKR}, although the idea is then implemented in a different, in a sense dual, way and some technicalities present there will not be needed in the present argument.\smallskip

Of course, Theorem B ought to be compared to Kunen's celebrated construction, still under CH, of a non-separable Banach space that contains no uncountable biorthogonal system (\cite{Kunen note1, Kunen note2}, see \cite[\S 7]{Negr84}). On the one hand, the existence of a non-separable Banach space without uncountable biorthogonal systems is a much stronger result than the mere existence of some dense subspace without such systems; on the other one, our result applies to a rather large class of well-behaved Banach spaces.

In particular, the most striking and perhaps unexpected consequence of the above theorem is the existence of a dense subspace of the non-separable Hilbert space $\ell_2(\omega_1)$ that contains no uncountable biorthogonal system; this also offers a substantial strengthening of Gudder's result \cite{Gudder} concerning the existence of non-separable inner product spaces with no uncountable orthonormal system (see Section \ref{Sec: IPS} fore more details).\smallskip

When comparing Theorem A and Theorem B, the former admitting a stronger version in the case ${\rm dens}\,X\leq\mathfrak{c}$ and the latter only stated for such a case, it is natural to wonder if Theorem B can be extended to larger Banach spaces. In Section \ref{Sec: large}, we shall discuss the situation for WLD Banach spaces with density character at least $\mathfrak{c}^+$ and we shall show that the above results cannot be improved. Among others, we prove the following theorem.
\begin{theoremC} Let $X$ be a WLD Banach space with ${\rm dens}\,X=\mathfrak{c}^+$. Then, every dense subspace of $X$ contains a biorthogonal system of cardinality $\mathfrak{c}^+$.
\end{theoremC}

In conclusion to this section, let us mention that we do not know if some version of Theorem B can be valid without extra set-theoretical assumptions; in particular, we do not know if there exists, in ZFC, a non-separable normed space without uncountable biorthogonal systems.  This seems to be closely related to the problem of understanding how much of the arguments in \cite{Todorcevic} can be performed in absence of completeness (see \cite{HJ}). Moreover, we do not know what happens when trying to extend Theorem A beyond the WLD case. For example, we don't know if Theorem A is true for $\ell_\infty$.

\section{Preliminaries}
Our notation concerning Banach space theory is standard, for example, as in \cite{ak} or \cite{FHHMZ}. The unique difference, though a cardinal one, is that by \emph{subspace} of a normed (or Banach) space we understand a linear subspace, not necessarily a closed one. When closedness is assumed, it will be stressed explicitly.

It will be convenient to adopt von Neumann's definition of ordinal numbers and to regard cardinal numbers as initial ordinal numbers. In particular, we write $\omega$ for $\aleph_0$, $\omega_1$ for $\aleph_1$, \emph{etc}., as we often view cardinal numbers as well-ordered sets; we denote by $\mathfrak{c}$ the cardinality of continuum. For a cardinal number $\kappa$, we write $\kappa^+$ for the smallest cardinal number that is strictly greater than $\kappa$. Moreover, for cardinal numbers $\kappa$ and $\lambda$, the notation $\kappa^\lambda$ will always refer to cardinal exponentiation. The cardinality of a set $A$ will be denoted $|A|$. Occasionally, we also write $\N=\{1,2,\dots\}$ and $\N_0=\N\cup\{0\}$. We refer to \cite{Jech} or \cite{kunen} for more on set-theoretical background.

\subsection{WLD Banach spaces}\label{Sec: WLD}
Most our techniques in this note will depend on the possibility to introduce suitable systems of coordinates in some classes of Banach spaces; we shall remind here the notions relevant to our paper.

Given a normed space $X$, a system $\{x_\gamma;x_\gamma^*\} _{\gamma\in\Gamma}\subseteq X\times X^*$ is a \emph{biorthogonal system} for $X$ if $\langle x_\alpha^*, x_\beta\rangle =\delta_{\alpha,\beta}$, whenever $\alpha,\beta\in\Gamma$. A biorthogonal system  $\{x_\gamma;x_\gamma^*\}_{\gamma\in\Gamma}$ is a \emph{Markushevich basis} (\emph{M-basis}, for short) if
$$\overline{\rm span}\{x_\gamma\}_{\gamma\in\Gamma}=X\qquad\text{and}\qquad \overline{\rm span}^{w^*}\{x_\gamma^*\}_{\gamma\in\Gamma}=X^*.$$

In the context of separable normed spaces, it is a classical Markushevich's result \cite{Markushevich} that every separable normed space admits an M-basis. Moreover, several classes of (mainly, non-separable) Banach spaces can be characterised by the existence of M-bases with certain additional properties, \emph{cf.} \cite[Chapter 6]{HMVZ}; in particular, WLD Banach spaces admit an handy such characterisation, that we shall use as equivalent way to introduce them.\smallskip

Given an M-basis $\{x_\gamma;x_\gamma^*\}_{\gamma \in\Gamma}$ for a Banach space $X$, the \emph{support} of a functional $x^*\in X^*$ is the set
$${\rm supp}\, x^*:=\{\gamma\in\Gamma\colon \langle x^*, x_\gamma\rangle \neq0\}.$$
A functional $x^*\in X^*$ is said to be \emph{countably supported} by $\{x_\gamma;x_\gamma^*\}_{\gamma \in\Gamma}$, or $\{x_\gamma;x_\gamma^*\}_{\gamma \in\Gamma}$ \emph{countably supports} $x^*$, when its support is a countable subset of $\Gamma$. A Banach space $X$ is \emph{weakly Lindel\"of determined} (\emph{WLD}, for short) if it admits an M-basis $\{x_\gamma;x_\gamma^*\}_{\gamma\in\Gamma}$ that countably supports $X^*$, \emph{i.e.}, every $x^*\in X^*$ is countably supported by $\{x_\gamma; x_\gamma^*\}_{\gamma\in\Gamma}$.

Most our arguments will only require the existence of an M-basis that countably supports the dual space; in one instance, however, we shall need the result that every WLD Banach space admits a \emph{projectional resolution of the identity} \cite{Valdivia}, see, \emph{e.g.}, \cite[p. 180]{HMVZ}.\smallskip

The class of WLD Banach spaces has been widely investigated in the literature and many its properties have been detected due to the efforts of several mathematicians; let us refer to \cite[\S VI.7]{DGZ}, \cite[\S 14.5]{FHHMZ}, \cite[\S 3.4, \S 5.4]{HMVZ} \cite[\S 19.8]{KKLP}, \cite{Kalenda survey}, \cite{Ziz03} and the references therein for more details.

\subsection{Orthonormal systems in inner product spaces}\label{Sec: IPS}
In this part, we shall collect some known results on orthonormal systems in inner product spaces, that we shall compare to our results.

It is easy to see that, for an orthonormal system $\{e_\gamma\}_{\gamma\in\Gamma}$ in an inner product space $H$, the following conditions are equivalent:
\begin{romanenumerate}
\item $\{e_\gamma\}_{\gamma\in\Gamma}$ is complete, \emph{i.e.}, $\overline{\rm span}\{e_\gamma\} _{\gamma\in\Gamma}=H$;
\item $\{e_\gamma\}_{\gamma\in\Gamma}$ is a basis for $H$;
\item Parseval's equality $\|x\|^2=\sum_{\gamma\in\Gamma} |\langle e_\gamma,x\rangle|^2$ holds true for every $x\in H$.
\end{romanenumerate}
Moreover, these conditions obviously imply that $\{e_\gamma\}_{\gamma\in\Gamma}$ is maximal, the converse being in general false in absence of completeness. As a simple example, let $\{e_j\}_{j=1}^\infty$ be a complete orthonormal system in a Hilbert space $H$, let
$$V:={\rm span}\left\{\sum_{j=1}^\infty 2^{-j}e_j, e_2,e_3,\dots\right\}$$
and observe that the maximal orthonormal system $\{e_j\}_{j=2}^\infty$ is not complete in $V$. More generally, every non complete inner product space contains a maximal orthonormal system that is not complete, \cite{Gudder2}.\smallskip

Moreover, it is a standard fact that every separable inner product space admits a complete orthonormal system, due to the Gram-Schmidt algorithm. This is known to be false when the separability assumption is dropped. More precisely, it was discovered by several authors independently, see, \emph{e.g.}, \cite{Farah, Fuchino, Shelah}, that for every uncountable cardinality $\lambda$ there exists an inner product space $H$ with ${\rm dens}\,H=\lambda$ that admits no complete orthonormal system.

An earlier result in a similar direction was proved by Gudder \cite{Gudder}, who constructed a non-separable inner product space that contains no uncountable orthonormal system (see also \cite[Problem 54]{Halmos}). In order to state a more general result, let us recall the elementary fact that any two maximal orthonormal sets in an inner product space $H$ must have the same cardinality, denoted ${\rm dim}\,H$. Then, given cardinal numbers $\kappa$ and $\lambda$ with $\kappa\leq\lambda$, there exists an inner product space $H$ with ${\rm dim}\,H=\kappa$ and ${\rm dens}\,H=\lambda$ if and only if $\lambda\leq\kappa^\omega$, \cite{BCW}; let us also refer to \cite{Farah, Fuchino} for a discussion of the result.

\subsection{Perturbation of M-bases} In this part, we shall show that in the class of separable Banach spaces Problem \ref{Problem} has a positive answer indeed. As we already mentioned, we consider this a folklore result, whose proof is essentially the same as in the Small Perturbation principle, more precisely its extension to M-bases (see, \emph{e.g.}, \cite[p. 46]{LiTzaI}). For the sake of completeness, we shall give the proof.

\begin{theorem}\label{Th: densely for separable} Let $X$ be a separable Banach space and let $\{e_j;e^*_j\}_{j=1}^\infty$ be an M-basis for $X$. Then every dense subspace of $X$ contains a further dense subspace that is isomorphic to ${\rm span}\{e_j\} _{j=1}^\infty$.

In particular, every two dense subspaces of $X$ are densely isomorphic.
\end{theorem}
\begin{proof} Let $Y$ be any dense subspace $X$ and select vectors $y_j\in Y$ such that $\|y_j-e_j\|\cdot\|e^*_j\|<2^{-(j+1)}$, for each $j\in\N$. Let us now consider the linear operator $T\colon {\rm span}(e_j) _{j=1}^\infty\to{\rm span}(y_j)_{j=1}^\infty$ defined by the rule $Te_j=y_j$ ($j\in\N$). For each $x\in{\rm span}(e_j) _{j=1}^\infty$, we then have (noting that the first series is indeed a finite sum)
$$\|x-Tx\|=\left\|\sum_{j=1}^\infty \langle e^*_j,x\rangle (e_j-y_j) \right\|\leq \sum_{j=1}^\infty\|e^*_j\|\|e_j-y_j\|\cdot\|x\| \leq \sum_{j=1}^\infty2^{-(j+1)}\|x\|=\frac{1}{2}\|x\|.$$
Therefore, $\frac{1}{2}\|x\|\leq\|Tx\|\leq\frac{3}{2}\|x\|$, whence $T$ is an isomorphism between ${\rm span}(e_j) _{j=1}^\infty$ and $Y_0:= {\rm span}(y_j)_{j=1}^\infty\subseteq Y$.

Finally, we check that $Y_0$ is a dense subspace of $X$. If this were not true, by Riesz' lemma we would find a unit vector $x\in S_X$ with ${\rm dist}(x,Y_0)>1/2$; by approximation, we could also assume $x\in {\rm span}(e_j) _{j=1}^\infty$. But then, $\|x-Tx\|\leq1/2$ and $Tx\in Y_0$ would yield a contradiction.
\end{proof}

\begin{remark} It is worth observing that the canonical dense subspace that we built in Theorem \ref{Th: densely for separable} is moreover unique up to isomorphisms. Indeed, Sophie Grivaux \cite[\S 2]{Grivaux} proved that any two countably dimensional, dense subspaces of a Banach space are isomorphic. We are grateful to Gilles Godefroy for pointing out to us this result.
\end{remark}

As an immediate consequence of the theorem and the considerations before the statement of Problem \ref{Problem}, we obtain the following corollary.
\begin{corollary} Two separable normed spaces are densely isomorphic if and only if their completions are isomorphic.
\end{corollary}

It is well known that the canonical basis of the space $\ell_1(\Gamma)$ is stable under more drastic perturbations than the ones permitted in the Small Perturbation lemma. This fact allows us to obtain a positive answer to the problem also for the spaces $\ell_1(\Gamma)$.
\begin{theorem}\label{Th: densely for ell1} Let $(e_\gamma)_{\gamma\in\Gamma}$ be the canonical basis of $\ell_1(\Gamma)$. Then, every dense subspace of $\ell_1(\Gamma)$ contains a dense subspace isomorphic to ${\rm span}\{e_\gamma\}_{\gamma\in\Gamma}$.

In particular, any two dense subspaces of $\ell_1(\Gamma)$ are densely isomorphic.
\end{theorem}

\begin{proof} Fixed a dense subspace $Y$ of $\ell_1(\Gamma)$ and $\e\in(0,1)$, we may find vectors $(y_\gamma)_{\gamma\in\Gamma}\subseteq Y$ such that $\|e_\gamma - y_\gamma\|<\e$ ($\gamma\in\Gamma$). By the standard perturbation properties of the $\ell_1$ basis (see, \emph{e.g.}, \cite[p. 331]{Jameson}), we derive that $(y_\gamma)_{\gamma\in\Gamma}$ is a Schauder basis on $\ell_1(\Gamma)$, equivalent to $(e_\gamma)_{\gamma\in\Gamma}$. Hence, ${\rm span}(y_\gamma)_{\gamma\in\Gamma}$ is a dense subspace of $Y$, isomorphic to ${\rm span}(e_\gamma) _{\gamma\in\Gamma}$.
\end{proof}

\section{Non-separable WLD Banach spaces}\label{Sec: Th A}
In this section, we shall focus on dense subspaces of non-separable WLD Banach spaces and we shall show how to build two of them that are not densely isomorphic, thereby proving Theorem A. We shall first consider the case when the density character of the Banach space $X$ under consideration satisfies $\omega_1\leq {\rm dens}\,X\leq \mathfrak{c}$, in which case we prove a stronger result. We then derive the validity of the general result from this particular case.\smallskip

Let us start with two well-known lemmata that we shall exploit in the course of the argument. Their proofs are so simple that we include them here.

\begin{lemma}\label{lemma: conv WLD} Let $\{e_\alpha;e^*_\alpha\}_{\alpha\in\Gamma}$ be an M-basis for a non-separable WLD Banach space $X$. Then
$$0\in \overline{\rm conv}\{e_\alpha\}_{\alpha\in\Gamma}.$$
\end{lemma}
\begin{proof} If not, the Hahn-Banach theorem would yield us the existence of a functional $\p\in X^*$ and $\e>0$ such that $\langle\p,x\rangle\geq \e$, for each $x\in \overline{\rm conv}\{e_\alpha\}_{\alpha\in\Gamma}$. In particular, $\langle\p,e_\alpha\rangle\geq\e$, for each $\alpha\in\Gamma$, whence ${\rm supp}\,\p=\Gamma$ is uncountable, a contradiction.
\end{proof}

The next lemma, due to Victor Klee \cite{Klee} (see also \cite[Proposition 2.1.6]{GurariyLusky}, or \cite[p.~113]{Milman}), is at the origin of the theory of overcomplete sequences. We need one piece of notation prior to its statement (which we only formulate in the generality needed here).

Given a normalised M-basis $\{e_j;e^*_j\}_{j=0}^\infty$ for a Banach space $X$ and $q\in(0,1)$, we set
$$g_q:=\sum_{j=0}^\infty q^j e_j.$$
\begin{lemma}[\cite{Klee}]\label{lemma: overcomplete} For every infinite set $J\subseteq(0,1/2)$, one has
$$\overline{\rm span}\{g_q\}_{q\in J}=X.$$
\end{lemma}
\begin{proof} If not, we would find a non-zero functional $\p\in X^*$ such that $\langle\p,g_q\rangle=0$, for each $q\in J$, \emph{i.e.},
$$\sum_ {j=0}^\infty\langle\p,e_j\rangle\cdot q^j=0 \qquad (q\in J).$$
Therefore, the function
$$q\mapsto \sum_ {j=0}^\infty\langle\p,e_j\rangle\cdot q^j$$
is an analytic function on $(-1,1)$ and it equals $0$ on infinitely many points in $[0,1/2]$ (namely, each point in $J$). According to the identity principle, it vanishes identically, whence $\langle\p,e_j\rangle=0$ for each $j\in\N_0$. Consequently, $\p=0$, a contradiction.
\end{proof}

We are now ready for the first main result of the section.
\begin{theorem}\label{th: WLD small} Let $X$ be a WLD Banach space with $\omega_1\leq {\rm dens}\,X\leq\mathfrak{c}$. Then there exist two dense subspaces $Y$ and $Z$ of $X$ such that no non-separable subspace of $Y$ is isomorphic to a subspace of $Z$ (and vice versa).

A fortiori, $Y$ and $Z$ are not densely isomorphic.
\end{theorem}

\begin{proof} Let us denote by $\Gamma$ the cardinal number ${\rm dens}\,X$ and select a normalised M-basis $\{e_\alpha;e^*_\alpha\}_{\alpha<\Gamma}$ for $X$. We also pick an injective long sequence $(q_\alpha)_{\omega\leq \alpha<\Gamma}$ of scalars in $(0,1/2)$ with the property that every open subset of $(0,1/2)$ contains uncountably many scalars $q_\alpha$ (this is possible since the Euclidean topology of $(0,1/2)$ has a countable base).

Let us then consider the vectors
$$\tilde{e}_\alpha:=e_\alpha+\sum_{j=0}^\infty(q_\alpha)^j\cdot e_j \qquad (\omega\leq\alpha<\Gamma).$$
We claim that the subspace $Y:={\rm span}\{\tilde{e}_\alpha\}_{\omega \leq\alpha<\Gamma}$ is the first subspace we are seeking. To this aim, we first show that $Y$ is a dense subspace of $X$.

\begin{claim}\label{claim: span dense} $\overline{\rm span}\{\tilde{e}_\alpha\}_{\omega\leq\alpha<\Gamma}=X$.
\end{claim}

\begin{proof}[Proof of Claim \ref{claim: span dense}] We shall keep the notation from Lemma \ref{lemma: overcomplete} to denote by $g_q$ the vector
$$g_q:=\sum_{j=0}^\infty q^j e_j\qquad(q\in(0,1/2)).$$
In particular, for $\omega\leq\alpha<\Gamma$, we have
\begin{equation}\label{eq: split with geom}
    \tilde{e}_\alpha = g_{q_\alpha} + e_\alpha.
\end{equation}\smallskip

Let us now fix $q\in(0,1/2)$ and $\e>0$. By construction, there exists an uncountable subset $\Gamma_{q,\e}$ of $[\omega,\Gamma)$ such that $|q - q_\alpha|<\e$, whenever $\alpha\in \Gamma_{q,\e}$. For each such $\alpha$, we have
\begin{equation}\label{eq: geom close}
\begin{split}
\left\|g_q - g_{q_\alpha}\right\|=\left\|\sum_{j=0}^\infty \left(q^j - (q_\alpha)^j\right)\cdot e_j\right\|\leq \sum_{j=0}^\infty \left\vert q^j - (q_\alpha)^j\right\vert \\
=\left\vert \sum_{j=0}^\infty \left(q^j - (q_\alpha)^j\right)\right\vert
=\left\vert \frac{q_\alpha - q}{(1-q)(1-q_\alpha)}\right\vert \leq 4\e.
\end{split}
\end{equation}
Thus, every convex combination of the vectors $(g_{q_\alpha})_{\alpha\in \Gamma_{q,\e}}$ has distance at most $4\e$ from $g_q$.\smallskip

On the other hand, $\{e_\alpha\}_{\alpha\in\Gamma_{q,\e}}$ is clearly an M-basis for its closed linear span, a non-separable WLD Banach space. Consequently, Lemma \ref{lemma: conv WLD} yields the existence of a convex combination of the vectors $\{e_\alpha\}_{\alpha\in\Gamma_{q,\e}}$ with arbitrarily small norm: there are positive scalars $\lambda_1,\dots,\lambda_n$ with $\sum_{i=1}^n\lambda_i=1$ and indices $\alpha_1,\dots,\alpha_n\in \Gamma_{q,\e}$ such that
$$\left\Vert \sum_{i=1}^n\lambda_i e_{\alpha_i} \right\Vert <\e.$$

If we consider the corresponding convex combination of the $\tilde{e}_\alpha$'s, by (\ref{eq: split with geom}) and (\ref{eq: geom close}), we then obtain
$$\left\Vert \sum_{i=1}^n\lambda_i \tilde{e}_{\alpha_i} -g_q\right\Vert = \left\Vert \sum_{i=1}^n\lambda_i e_{\alpha_i} + \sum_{i=1}^n\lambda_i g_{q_{\alpha_i}} - \sum_{i=1}^n\lambda_i g_q\right\Vert$$
$$\leq \e + \sum_{i=1}^n \lambda_i \left\Vert g_{q_{\alpha_i}}-g_q \right\Vert\leq5\e.$$
As a consequence of this argument (and $\e>0$ being arbitrary), we conclude that
\begin{equation}\label{eq: geom in span}
    \{g_q\}_{q\in(0,1/2)} \subseteq \overline{\rm conv}\{\tilde{e}_\alpha\}_{\omega\leq\alpha<\Gamma}.
\end{equation}\smallskip

This and (\ref{eq: split with geom}) then immediately imply that $e_\alpha\in \overline{\rm span}\{\tilde{e}_\alpha\} _{\omega\leq\alpha<\Gamma}$, whenever $\omega\leq\alpha<\Gamma$. Finally, Lemma \ref{lemma: overcomplete} implies that
$$\overline{\rm span}\{g_q\}_{q\in(0,1/2)}= \overline{\rm span}\{e_\alpha\}_{\alpha<\omega},$$
which, in conjunction with (\ref{eq: geom in span}), assures us that $e_\alpha\in \overline{\rm span}\{\tilde{e}_\alpha\} _{\omega\leq\alpha<\Gamma}$ also when $\alpha<\omega$, thereby concluding the proof of the claim.
\end{proof}

Before we find the second dense subspace $Z$ of $X$, let us also note the following crucial property of $Y$.
\begin{fact}\label{fact: functionals separate} The functionals $(e^*_j)_{j=0}^\infty$ separate points of $Y$.
\end{fact}

\begin{proof}[Proof of Fact \ref{fact: functionals separate}] Let $y\in Y$ be such that $\langle e^*_j,y\rangle=0$ for every $j\in\N_0$ and let us write $y=\sum_{i=0}^n a_i\,\tilde{e}_{\alpha_i}$, where $\{\alpha_0,\dots,\alpha_n\}\subseteq[\omega,\Gamma)$. Then, for each $j=0,\dots,n$,

$$0=\langle e^*_j,y\rangle=\sum_{i=0}^n a_i(q_{\alpha_i})^j;$$
in matrix form, the above equations read

$$\begin{pmatrix} 1 & 1 &\dots& 1 \\ q_{\alpha_0} & q_{\alpha_1} &\dots& q_{\alpha_n} \\ \vdots &\vdots&& \vdots\\ (q_{\alpha_0})^n & (q_{\alpha_1})^n &\dots& (q_{\alpha_n})^n \end{pmatrix} \cdot \begin{pmatrix}a_0\\a_1\\ \vdots\\a_n\end{pmatrix} = \begin{pmatrix}0\\ \vdots\\0\end{pmatrix}.$$

By the classical result that Vandermonde matrices are non-singular\footnote{Notice that the result is the finite-dimensional counterpart to Lemma \ref{lemma: overcomplete}. Indeed, a quick proof follows the same argument as in the lemma: if the determinant were null, there would be a non-zero vector in $\R^{n+1}$ orthogonal to each column of the matrix. But this would yield a non-zero polynomial of degree $n$ with $n+1$ roots, a contradiction.}, it then follows $a_0=\dots=a_n=0$, whence $y=0$.
\end{proof}

Let us now pass to the description of the second dense subspace $Z$ of $X$,
which is merely $Z:={\rm span}\{e_\alpha\}_{\alpha<\Gamma}$ (although we could equally well consider the linear span of any M-basis of $X$). The property we shall need of the subspace $Z$ is proved in the forthcoming lemma.

\begin{lemma}\label{lemma: no sequence separates} Let $\{v_\alpha;\p_\alpha\}_{\alpha\in\Gamma}$ be an M-basis for a WLD Banach space $X$, set $Z:={\rm span}\{v_\alpha\}_{\alpha\in\Gamma}$ and let $Z_0$ be a non-separable subspace of $Z$. Then no sequence of functionals separates points of $Z_0$.
\end{lemma}

\begin{proof}[Proof of Lemma \ref{lemma: no sequence separates}] Let us assume, by contradiction, that there exists a sequence $(\psi_j)_{j=1}^\infty$ of functionals that separates points of $Z_0$. Let us also select an uncountable linearly independent subset $(w_\beta)_{\beta<\omega_1}$ of $Z_0$. By definition, we may find, for each $\beta<\omega_1$, a finite subset $F_\beta$ of $\Gamma$ and scalars $(w_\beta^\alpha)_{\alpha\in F_\beta}$ such that
$$w_\beta =\sum_{\alpha\in F_\beta}w_\beta^\alpha\,v_\alpha.$$

The $\Delta$-system lemma (\cite[Lemma III.2.6]{kunen}) allows us to assume, up to passing to an uncountable subset of $\omega_1$, the existence of a finite subset $\Delta$ of $\Gamma$ such that $F_\beta\cap F_\gamma=\Delta$ whenever $\beta$ and $\gamma$ are distinct ordinals, smaller than $\omega_1$. Then the sets $(F_\beta\setminus\Delta)_{\beta<\omega_1}$ are evidently disjoint and non-empty, due to the linear independence of the vectors $w_\beta$. Consequently, when we denote by $N$ the countable set $N:=\cup_{j=1}^\infty {\rm supp}\,\psi_j$, we see that only countably many sets $F_\beta\setminus\Delta$ can intersect $N$. Therefore, up to passing to an uncountable subset of $\omega_1$, we can additionally assume that
\begin{equation}\label{eq: out of Delta}
(F_\beta\setminus\Delta)\cap N=\emptyset\qquad (\beta<\omega_1).
\end{equation}

Let us now consider the vectors
$$w_\beta\flag_\Delta:=\sum_{\alpha\in \Delta}w_\beta^\alpha\,v_\alpha\in {\rm span}\{v_\alpha\}_{\alpha\in\Delta};$$
since these uncountably many vectors belong to a finite-dimensional vector space, we may express the zero vector as a linear combination thereof. In other words, there exists a vector $w\in {\rm span}\{w_\beta\}_{\beta<\omega_1}$, which is non-zero (due to the linear independence of the vectors $w_\beta$) and $w\flag_\Delta=0$. In light of (\ref{eq: out of Delta}), we then derive that $w\flag_N=0$, whence $\langle\psi_j,w\rangle=0$, for each $j\in\N$, which is the desired contradiction.
\end{proof}

It is now immediate to conclude the proof, by checking that the subspaces $Y$ and $Z$ have no non-separable subspace in common. Indeed, Fact \ref{fact: functionals separate} yields that every subspace of $Y$ admits a sequence of functionals that separates points, while no non-separable subspace of $Z$ does, in light of Lemma \ref{lemma: no sequence separates}. Therefore, no two non-separable subspaces of $Y$ and $Z$ respectively can be isomorphic.
\end{proof}

Before we proceed with the second theorem of the section, let us add a few comments on the above argument. First of all, it is evident that the system $\{\tilde{e}_\alpha,e^*_\alpha\} _{\omega\leq\alpha<\Gamma}$ constructed above is not an M-basis for $X$; this is clear from Lemma \ref{lemma: no sequence separates}, or from the fact that the functionals $\{e^*_\alpha\} _{\omega\leq\alpha<\Gamma}$ do not separate points of $X$. On the other hand, $\{\tilde{e}_\alpha,e^*_\alpha\} _{\omega\leq\alpha<\Gamma}$ is evidently an M-basis for $Y$ (and, in particular, $Y$ contains an uncountable biorthogonal system).\smallskip

Let us also observe that, although the non-separable subspaces of $Y$ and $Z$ are mutually non-isomorphic, it is easy to construct separable isomorphic subspaces of $Y$ and $Z$. More precisely, every separable subspace of $Z$ is isomorphic to a subspace of $Y$ and vice versa, as we prove in the fact below.
\begin{fact} Let $Y$ and $Z$ be the subspaces constructed in Theorem \ref{th: WLD small}. Then every separable subspace of $Z$ is isomorphic to a subspace of $Y$ and, vice versa, every separable subspace of $Y$ is isomorphic to a subspace of $Z$.
\end{fact}

\begin{proof} Let $Z_0$ be a separable subspace of $Z$. We first observe that $Z_0$ admits a countable Hamel basis. Indeed, if $(\bar{z}_n)_{n=1}^\infty$ is a dense sequence in $Z_0$, $N:=\cup_{n=1}^\infty{\rm supp}\,\bar{z}_n$ is a countable subset of $\Gamma$ and every element of $Z_0$ has support contained in $N$. Therefore, $Z_0\subseteq {\rm span}\{e_\alpha\}_{\alpha\in N}$.

Secondly, according to the classical Markushevich theorem (see, \emph{e.g.}, \cite[Lemma 1.21]{HMVZ}), we can select an M-basis for $Z_0$ that linearly spans $Z_0$: there exists an M-basis $\{z_n;z^*_n\}_{n=1}^\infty$ for $Z_0$ such that $Z_0={\rm span}\{z_n\}_{n=1}^\infty$. We then argue as in the proof of Theorem \ref{Th: densely for separable}. Select vectors $(y_n)_{n=1}^\infty$ in $Y$ such that $\|z_n-y_n\|\cdot\|z^*_n\| \leq2^{-(n+1)}$ ($n\in\N$) and observe that the correspondence $z_n\mapsto y_n$ defines an isomorphism from $Z_0$ onto ${\rm span}\{y_n\}_{n=1}^\infty\subseteq Y$.\smallskip

The proof of the `vice versa' assertion is essentially identical: if $Y_0$ is a separable subspace of $Y$ and $(y_n)_{n=1}^\infty$ is a dense sequence in $Y_0$, there is a countable set $N\subseteq\Gamma$ such that $(y_n)_{n=1}^\infty\subseteq{\rm span}\{\tilde{e}_\alpha\} _{\alpha\in N}$. Therefore, for each $y\in Y_0$, $\langle e^*_\alpha, y\rangle=0$ whenever $\omega\leq\alpha<\Gamma$ and $\alpha\notin N$. Since every vector in $Y$ is a linear combination of the vectors $\{\tilde{e}_\alpha\}_{\omega\leq\alpha<\Gamma}$, it follows that $Y_0\subseteq{\rm span}\{\tilde{e}_\alpha\} _{\alpha\in N}$ and we can argue as in the previous case.
\end{proof}

We shall next deduce from Theorem \ref{th: WLD small} the validity of Theorem A for every uncountable cardinality. Let us repeat the statement of the result under consideration, for convenience of the reader.
\begin{theorem}\label{Th: WLD large} Let $X$ be a non-separable WLD Banach space. Then there are dense subspaces $Y$ and $Z$ of $X$ that are not densely isomorphic.
\end{theorem}

\begin{proof} Since the case when ${\rm dens}\,X=\omega_1$ follows from Theorem \ref{th: WLD small}, we can assume that $\Gamma:={\rm dens}\,X\geq\omega_2$. Moreover, $X$ admits a projectional resolution of the identity, whence we can select a complemented (closed) subspace $X_0$ of $X$, with ${\rm dens}\,X_0=\omega_1$. Let us select a bounded projection $P$ from $X$ onto $X_0$ and write $X\simeq X_0\oplus X_1$, where $X_1=\ker P$. Finally, let us fix an M-basis $\{e_\alpha;e^*_\alpha\}_{\alpha<\Gamma}$ for $X$.

Following the argument in the proof of Theorem \ref{th: WLD small} (in particular, Claim \ref{claim: span dense} and Fact \ref{fact: functionals separate}), we can select a dense subspace $\tilde{Y}$ of $X_0$ that admits a sequence of functionals that separates points. We then set $Y:={\rm span}\{\tilde{Y},X_1\}$ and $Z:={\rm span}\{e_\alpha\}_{\alpha<\Gamma}$. Evidently, $Y$ and $Z$ are dense subspaces of $X$; moreover, it is plain that $P(Y)=\tilde{Y}$. Therefore, if $Y_0$ is any dense subspace of $Y$, then $P(Y_0)$ is a dense subspace of $\tilde{Y}$. In turn, this yields that $P(Y_0)$ is a non-separable subspace of $Y$ and that there exists a sequence of functionals that separates points of $P(Y_0)$.

Finally, assume by contradiction that $Y$ and $Z$ are densely isomorphic. Then, there are dense subspaces $Y_0$ of $Y$ and $Z_0$ of $Z$ and an isomorphism $T\colon Y_0\to Z_0$. Therefore, $T (P(Y_0))$ is a non-separable subspace of $Z$ and it admits a separating sequence of functionals. However, this is in contradiction with Lemma \ref{lemma: no sequence separates} and such contradiction completes the proof.
\end{proof}

\section{Uncountable biorthogonal systems}\label{Sec: Th B}
This section is dedicated to an alternative construction, under the assumption of the Continuum Hypothesis, of two non densely isomorphic, dense subspaces of a WLD Banach space with density character $\omega_1$. In the first and main part of the section, we shall prove Theorem B, whose statement is recalled below; in the second part, we show why the present result offers a more precise version of Theorem A.
\begin{theorem}[CH]\label{Th: CH No biorth} Let $X$ be a WLD Banach space with ${\rm dens}\,X=\omega_1$. Then there exists a dense subspace $Y$ of $X$ that contains no uncountable biorthogonal system.
\end{theorem}

\begin{proof} Let us select an M-basis $\{e_\alpha;e^*_\alpha\} _{\alpha<\omega_1}$ for $X$ and assume $\|e_\alpha\|=1$ for each $\alpha<\omega_1$. We shall start by constructing a new family $(\tilde{e}_\alpha)_{\alpha<\omega_1}$ of vectors in $X$. Fix an injective long sequence $(\lambda_\alpha)_{\alpha<\omega_1}\subseteq (0,1)$ and choose, for every $\alpha<\omega_1$, an enumeration $\sigma_\alpha$ of $[0,\alpha)$, \emph{i.e.}, a bijection $\sigma_\alpha\colon \omega\to[0,\alpha)$ when $\alpha$ is infinite, or $\sigma_\alpha\colon |\alpha|-1\to[0,\alpha)$ when $\alpha$ is a finite ordinal. We may now define vectors $\tilde{e}_\alpha$ ($\alpha<\omega_1$) as follows:

$$\tilde{e}_0 := e_0;$$
$$\tilde{e}_\alpha := e_\alpha+ \sum_{k=0}^{|\alpha|-1} (\lambda_\alpha)^k e_{\sigma_\alpha(k)}\qquad(1\leq\alpha<\omega);$$
$$\tilde{e}_\alpha := e_\alpha+ \sum_{k=0}^{\infty}(\lambda_\alpha) ^k e_{\sigma_\alpha(k)} \qquad(\omega\leq\alpha<\omega_1).$$
Plainly, the vectors $\tilde{e}_\alpha$ ($\alpha<\omega_1$) constitute a linearly independent subset of $X$.\smallskip

The above enumerations may be chosen arbitrarily and the subsequent argument will not depend on any specific such choice. On the other hand, a substantial part of the argument to be presented will involve explaining how to properly choose the coefficients $\lambda_\alpha$. Prior to this, let us observe that the vectors $(\tilde{e}_\alpha)_{\alpha<\omega_1}$ span a linearly dense subset of $X$, regardless of the choice of the coefficients $\lambda_\alpha$.
\begin{fact}\label{fact: lin dense} ${\rm span}\{\tilde{e}_\alpha\}_{\alpha<\omega_1}$ is linearly dense in $X$.
\end{fact}

\begin{proof}[Proof of Fact \ref{fact: lin dense}] We shall actually prove, by transfinite induction, the stronger assertion that $\overline{\rm span}\{e_\alpha\}_{\alpha<\beta} = \overline{\rm span}\{\tilde{e}_\alpha\}_{\alpha<\beta}$ for every $\beta\leq\omega_1$. Let us therefore assume the validity of the above equality for every $\beta<\gamma$, where $\gamma\leq\omega_1$. In the case where $\gamma$ is a limit ordinal, we immediately derive that $\{e_\alpha\}_{\alpha<\gamma} \subseteq \overline{\rm span}\{\tilde{e}_\alpha\}_{\alpha<\gamma}$, whence $\overline{\rm span}\{e_\alpha\}_{\alpha<\gamma} = \overline{\rm span}\{\tilde{e}_\alpha\}_{\alpha<\gamma}$ follows (the `$\supseteq$' inclusion being trivial).

If, on the other hand, $\gamma$ is a successor ordinal, say $\gamma=\eta+1$, we have in particular $\overline{\rm span}\{e_\alpha\} _{\alpha<\eta} = \overline{\rm span}\{\tilde{e}_\alpha\}_{\alpha<\eta}$. By the very definition, we also have $\tilde{e}_\eta - e_\eta\in\overline{\rm span}\{e_\alpha\}_{\alpha<\eta}=\overline{\rm span}\{\tilde{e}_\alpha\} _{\alpha<\eta}$, whence $e_\eta\in\overline{\rm span}\{\tilde{e}_\alpha\} _{\alpha\leq\eta}$. Plainly, this yields $\overline{\rm span}\{e_\alpha\} _{\alpha\leq\eta} = \overline{\rm span}\{\tilde{e}_\alpha\}_{\alpha\leq\eta}$, namely $\overline{\rm span}\{e_\alpha\}_{\alpha<\gamma} = \overline{\rm span}\{\tilde{e} _\alpha\}_{\alpha<\gamma}$, which is the desired conclusion.
\end{proof}

We shall now describe how to choose the sequence $(\lambda_\alpha)_{\alpha<\omega_1}$. Since every M-basis in the WLD space $X$ countably supports $X^*$, the evaluation map
$$\p\mapsto (\langle\p,e_\alpha\rangle)_{\alpha<\omega_1}$$
defines a bounded, linear injection of $X^*$ into $\ell_\infty^c(\omega_1)$. (By $\ell_\infty^c(\omega_1)$ we understand the closed subspace of $\ell_\infty(\omega_1)$ comprising all vectors in $\ell_\infty(\omega_1)$ with countable support.) Moreover, it is elementary to verify that $\ell_\infty^c(\omega_1)$ has cardinality the continuum. Indeed, let us denote by $\ell_\infty(\alpha)$ the subspace of $\ell_\infty^c(\omega_1)$ comprising vectors with support contained in $[0,\alpha)$ ($\alpha<\omega_1$). Evidently, $\ell_\infty^c(\omega_1)=\cup_{\alpha<\omega_1}\ell_\infty(\alpha)$ and $\ell_\infty(\alpha)$ is naturally isometric to $\ell_\infty$ (when $\omega\leq\alpha<\omega_1$), whence $|\ell_\infty(\alpha)|=\mathfrak{c}$; consequently, $|\ell_\infty^c(\omega_1)|=\mathfrak{c}$ too.

Therefore, we obtain that $|X^*|=\mathfrak{c}$ and the assumption of the Continuum Hypothesis allows us to well order $X^*$ in an injective $\omega_1$-sequence $(\p_\alpha)_{\alpha<\omega_1}$. This $\omega_1$-enumeration of $X^*$ permits us to choose the parameters $(\lambda_\alpha)_{\alpha<\omega_1}$, as follows.

\begin{claim}\label{claim: nonzero det} It is possible to choose the parameters $(\lambda_\alpha)_{\alpha< \omega_1}$ in such a way that the following assertion is satisfied.
\begin{description}
\item[(NS)] For every $N\in\N$ and for every choice of ordinal numbers $\alpha_1, \dots,\alpha_N<\omega_1$ and $\beta_1,\dots,\beta_N<\omega_1$ with the properties that:
\begin{romanenumerate}
\item $\{\p_{\alpha_1},\dots,\p_{\alpha_N}\}$ is a linearly independent set;
\item $\beta_1<\beta_2<\dots<\beta_N$;
\item $\alpha_1,\dots,\alpha_N<\beta_1$;
\item ${\rm supp}\,\p_{\alpha_1},\dots,{\rm supp}\,\p_{\alpha_N}<\beta_1$;
\end{romanenumerate}
one has:
$$\det\left(\left(\left\langle\p_{\alpha_i},\tilde{e}_{\beta_j}\right\rangle\right)_{i,j=1}^N\right)\neq0.$$
\end{description}
\end{claim}

\begin{proof}[Proof of Claim \ref{claim: nonzero det}]  We shall select $(\lambda_\alpha)_{\alpha<\omega_1}$ arguing by transfinite induction on $\gamma:=\beta_N<\omega_1$. Let us observe preliminarily that when $\gamma=0$, condition {\bf (NS)} is empty, while $\tilde{e}_0=e_0$, regardless of the choice of $\lambda_0$. Let us therefore assume to have already selected $(\lambda_\alpha) _{\alpha<\gamma}$, for some $\gamma<\omega_1$ in such a way that, for every $N\in\N$ and every choice of $\alpha_1,\dots,\alpha_N<\omega_1$ and $\beta_1,\dots,\beta_N<\omega_1$ satisfying properties (i)--(iv) and such that $\beta_N<\gamma$, the corresponding determinant appearing in {\bf (NS)} is non-zero. We shall select $\lambda_\gamma$ in a way that all the above determinants are non-zero, for every choice of $N\in\N$, $\alpha_1,\dots,\alpha_N<\omega_1$ and $\beta_1,\dots,\beta_N<\omega_1$ satisfying (i)--(iv) and with $\beta_N\leq\gamma$.\smallskip

Let us therefore fix $N\in\N$ and ordinal numbers $\alpha_1,\dots,\alpha_N<\omega_1$ and $\beta_1,\dots,\beta_N<\omega_1$ that satisfy conditions (i)--(iv) and with $\beta_N=\gamma$; notice that the parameters $\lambda_{\beta_1},\dots,\lambda_{\beta_{N-1}}$ have already been determined in the previous steps of the induction and, as such, we only have to choose the parameter appearing in $\tilde{e}_\gamma$. The determinant corresponding to the above choice of $N$, $\alpha_1,\dots,\alpha_N$ and $\beta_1,\dots,\beta_N$ can be evaluated by means of Laplace's expansion theorem on the last column:

$$\det\left(\left(\langle\p_{\alpha_i},\tilde{e}_{\beta_j}\rangle\right) _{i,j=1}^N\right)= \det \begin{pmatrix} \langle\p_{\alpha_1},\tilde{e}_{\beta_1}\rangle &\dots& \langle\p_{\alpha_1},\tilde{e}_{\beta_N}\rangle\\ \vdots && \vdots\\ \langle\p_{\alpha_N},\tilde{e}_{\beta_1}\rangle &\dots& \langle\p_{\alpha_N},\tilde{e}_{\beta_N}\rangle \end{pmatrix}=$$
$$\sum_{i=1}^N(-1)^{N+i}d_i \langle\p_{\alpha_i},\tilde{e}_{\beta_N}\rangle =\left\langle \sum_{i=1}^N(-1)^{N+i}d_i\p_{\alpha_i},\tilde{e}_{\beta_N} \right\rangle,$$
where $d_i$ is the determinant of the $(N-1)\times(N-1)$ matrix obtained removing the $i$-th row and the $N$-th column from the original matrix $\left(\langle\p_{\alpha},\tilde{e}_{\beta_j}\rangle\right)_{i,j=1}^N$. (In the case when $N=1$, the above formula is again true provided that we set $d_1=1$.) The transfinite induction assumption, applied to $\{\alpha_1,\dots,\alpha_N\}\setminus\{\alpha_i\}$ and $\{\beta_1,\dots, \beta_{N-1}\}$ then leads us to the conclusion that $d_i\neq0$, for each $i=1,\dots,N$. Thereby, (i) yields that the functional $\p:=\sum_{i=1}^N (-1)^{N+i}d_i\p_{\alpha_i}$ is non-zero. Moreover, ${\rm supp}\,\p<\beta_1<\gamma$, in light of (iv); consequently,

$$\det\left(\left(\langle\p_{\alpha_i},\tilde{e}_{\beta_j}\rangle\right) _{i,j=1}^N\right)=\langle\p,\tilde{e}_\gamma\rangle=\sum_{k=0}^\infty \langle\p,e_{\sigma_\gamma(k)}\rangle\cdot(\lambda_\gamma)^k. $$

(In the case where $\gamma<\omega$, the summation is extended from $0$ to $|\gamma|-1$.) When considered as a function of $\lambda_\gamma$, the above expression is a power series whose coefficients are bounded and not all of them are null. Therefore, the correspondence
$$\lambda_\gamma\mapsto \det\left(\left(\langle\p_{\alpha_i},\tilde{e} _{\beta_j}\rangle\right) _{i,j=1}^N\right)$$
defines a non-trivial analytic function on $(-1,1)$ and, accordingly, such function has at most countably many zeros on $(-1,1)$, in light of the identity principle. Moreover, due to (ii) and (iii), there are only countably many choices for $N\in\N$, $\alpha_1,\dots,\alpha_N$ and $\beta_1,\dots,\beta_N$ that satisfy $\beta_N=\gamma$ and (i)--(iv). Therefore, taking the union of all the corresponding countable zero sets of those countably many analytic functions, we obtain a countable set $\Lambda\subseteq(-1,1)$ such that every determinant appearing in {\bf (NS)} is non-zero, for each choice of $\lambda_\gamma\in(-1,1) \setminus\Lambda$. Hence, when we select $\lambda_\gamma\in(0,1) \setminus\Lambda$ and such that $\lambda_\gamma\neq\lambda_\alpha$ for each $\alpha<\gamma$, we obtain that {\bf (NS)} is satisfied also for $\gamma$; therefore, the transfinite induction is complete and so is the proof of Claim \ref{claim: nonzero det}.
\end{proof}

Having the claim proved, we may choose parameters $(\lambda_\alpha)_{\alpha<\omega_1}$ satisfying condition {\bf (NS)} above; we then denote by $Y:={\rm span}\{\tilde{e}_\alpha\} _{\alpha<\omega_1}$, where the vectors $(\tilde{e}_\alpha)_ {\alpha<\omega_1}$ are obtained from the presently chosen sequence $(\lambda_\alpha)_{\alpha<\omega_1}$. We shall conclude the proof by showing that $Y$ is the desired dense subspace with no uncountable biorthogonal system. The density of $Y$ in $X$ being already established in Fact \ref{fact: lin dense}, we only need to prove the following claim.
\begin{claim}\label{claim: no biorth system} The subspace $Y$ contains no uncountable biorthogonal system.
\end{claim}

\begin{proof}[Proof of Claim \ref{claim: no biorth system}] Assume, by contradiction, that $Y$ contains an uncountable biorthogonal system, say $\{u_\alpha;g_\alpha\}_{\alpha<\omega_1}$; obviously, we may assume that $g_\alpha\in X^*$, for each $\alpha<\omega_1$. Moreover, by definition, for each $\alpha<\omega_1$ there exist a finite subset $F_\alpha$ of $\omega_1$ and scalars $(u_\alpha^\beta)_{\beta\in F_\alpha}$ such that
$$u_\alpha = \sum_{\beta\in F_\alpha}u_\alpha^\beta\,\tilde{e}_\beta.$$

Application of the $\Delta$-system lemma to the system of finite sets $(F_\alpha)_{\alpha<\omega_1}$ allows us to assume, up to passing to an uncountable subset of $\omega_1$ and relabelling, that there exist $n\in\N$ and a finite subset $\Delta$ of $\omega_1$ such that:
\begin{enumerate}
    \item $|F_\alpha|=n-1$, for each $\alpha<\omega_1$;
    \item $F_\alpha\cap F_\beta=\Delta$ for distinct $\alpha,\beta<\omega_1$;
    \item $\Delta < F_\alpha\setminus\Delta < F_\beta\setminus\Delta$ whenever $\alpha<\beta<\omega_1$.
\end{enumerate}
Indeed, $(1)$ and $(2)$ follow directly from the $\Delta$-system lemma and $(3)$ is achieved by a simple transfinite induction argument, based on the fact that, the sets $(F_\alpha\setminus\Delta)_{\alpha<\omega_1}$ being disjoint, only countably many of them can intersect each set of the form $[0,\beta)$ ($\beta<\omega_1$). (Note that the sets $F_\alpha\setminus\Delta$ are indeed non-empty, by virtue of the linear independence of the vectors $\{u_\alpha\}_{\alpha<\omega_1}$.)\smallskip

Let us now consider the functionals $g_1,\dots,g_{n^2}$. Since we enumerated $X^*$ in the $\omega_1$-sequence $(\p_\alpha)_{\alpha<\omega_1}$, there exist ordinals $\alpha_1,\dots,\alpha_{n^2}$ such that $g_j=\p_{\alpha_j}$ ($j=1,\dots,n^2$). Moreover, each $g_j$ is countably supported, whence we can select a countable ordinal $\Gamma$ such that
\begin{romanenumerate}
    \item ${\rm supp}\,\p_{\alpha_j}\leq\Gamma$ for each $j=1,\dots,n^2$;
    \item $\alpha_j\leq\Gamma$ for each $j=1,\dots,n^2$;
    \item $n^2\leq\Gamma$ (which is actually consequence of (ii)).
\end{romanenumerate}
According to $(3)$, we are now in position to choose $\theta_1<\omega_1$ with the properties that $\Gamma<\theta_1$ and $\Gamma < F_{\theta_1}\setminus\Delta$. Let us also select ordinals $\theta_2,\dots,\theta_n$ such that $\Gamma<\theta_1<\theta_2<\dots<\theta_n$; then $(3)$ also ensures us that
\begin{equation}\label{eq: supports slide}
\Gamma < F_{\theta_1}\setminus\Delta < F_{\theta_2}\setminus\Delta <\dots< F_{\theta_n}\setminus\Delta.
\end{equation}\smallskip

Consider now the corresponding vectors $u_{\theta_1},\dots,u_{\theta_n}$ and write them as follows:

$$u_{\theta_i}=\sum_{\beta\in F_{\theta_i}}u_{\theta_i}^\beta\,\tilde{e}_\beta = \sum_{\beta\in \Delta}u_{\theta_i}^\beta\,\tilde{e}_\beta + \sum_{\beta\in F_{\theta_i}\setminus\Delta}u_{\theta_i}^\beta\,\tilde{e}_\beta =: u_{\theta_i}\flag_\Delta + u_{\theta_i}\flag_{\Delta^\complement}.$$
Evidently, $\{u_{\theta_i}\flag_\Delta\}_{i=1}^n\subseteq {\rm span}\{\tilde{e}_\beta\}_{\beta\in\Delta}$, a vector space of dimension at most $n-1$, according to $(1)$. Therefore, there exist scalars $c_1,\dots,c_n$, not all equal to $0$, such that
\begin{equation}\label{eq: null on Delta}
\sum_{i=1}^n c_i\, u_{\theta_i}\flag_\Delta=0.
\end{equation}\smallskip

Let us finally consider the non-zero vector $u:=\sum_{i=1}^n c_i\, u_{\theta_i}$; on the one hand, (iii) yields, for each $j=1,\dots,n^2$
$$\langle\p_{\alpha_j},u\rangle=\langle g_j,u\rangle=\sum_{i=1}^n c_i\, \langle g_j, u_{\theta_i}\rangle=0.$$
On the other hand, condition $(1)$ implies that $u$ is linear combination of at most $n^2$ $\tilde{e}_\beta$'s; moreover, (\ref{eq: supports slide}) and (\ref{eq: null on Delta}) assure us that $u$ only depends on those $\tilde{e}_\beta$ for which $\Gamma<\beta$. Consequently, we may find scalars $a_1,\dots,a_{n^2}$ and ordinals $\beta_1,\dots,\beta_{n^2}$ with $\Gamma<\beta_1<\beta_2<\dots<\beta_{n^2}<\omega_1$ and such that
$$u=\sum_{i=1}^{n^2}a_i\, \tilde{e}_{\beta_i}$$
(notice that we insist the linear combination to have length exactly $n^2$, at the cost of choosing some $a_i=0$). Therefore, the equations $\langle\p_{\alpha_j},u\rangle=0$ now read

$$\begin{pmatrix} \langle\p_{\alpha_1},\tilde{e}_{\beta_1}\rangle &\dots& \langle\p_{\alpha_1},\tilde{e}_{\beta_{n^2}}\rangle\\ \vdots && \vdots\\ \langle\p_{\alpha_{n^2}},\tilde{e}_{\beta_1}\rangle &\dots& \langle\p_{\alpha_{n^2}},\tilde{e}_{\beta_{n^2}}\rangle \end{pmatrix} \cdot \begin{pmatrix}a_1\\ \vdots\\a_{n^2}\end{pmatrix} = \begin{pmatrix}0\\ \vdots\\0\end{pmatrix}.$$

However, (i) and (ii) imply that the ordinals $\{\alpha_1,\dots,\alpha_{n^2}\}$ and $\{\beta_1,\dots,\beta_{n^2}\}$ satisfy the assumptions of Claim \ref{claim: nonzero det}, which therefore yields
$$\det\left((\langle\p_{\alpha_j},\tilde{e}_{\beta_i}\rangle)_{j,i=1}^{n^2} \right)\neq0,$$
whence $a_1=\dots=a_{n^2}=0$. This ultimately implies $u=0$, a contradiction which concludes the proof of the claim.
\end{proof}\end{proof}

We shall now show that Theorem B yields an alternative way to build dense subspaces that are not densely isomorphic. More precisely, it may be used to prove the existence of two dense subspaces with no isomorphic non-separable subspaces, thereby giving an alternative proof of Theorem \ref{th: WLD small}.

Let $X$ be a WLD Banach space with ${\rm dens}\,X=\omega_1$, assume CH, and let $Y$ be a dense subspace of $X$ that contains no uncountable biorthogonal system. In order to find a dense subspace $Z$ of $X$ such that no non-separable subspace of $Z$ is isomorphic to a subspace of $Y$, we only need to construct a dense subspace $Z$ of $X$ every whose non-separable subspace contains an uncountable biorthogonal system. This is obtained in the forthcoming lemma.

\begin{lemma}\label{Lemma: Biort System in span of M-basis} Let $\{e_\alpha;e^*_\alpha\}_{\alpha\in\Gamma}$ be an M-basis for a Banach space $X$. Then every non-separable subspace of $Z:={\rm span}\{e_\alpha\} _{\alpha\in\Gamma}$ contains an uncountable biorthogonal system.
\end{lemma}

\begin{proof} If $Z_0$ is any non-separable subspace of $Z$, then we can choose an uncountable linearly independent set $(v_\beta)_{\beta<\omega_1}\subseteq Z_0$. By definition, for each $\beta<\omega_1$, there exist a finite subset $F_\beta$ of $\Gamma$ and non-zero scalars $(v^\alpha_\beta)_{\alpha\in F_\beta}$ such that
$$v_\beta=\sum_{\alpha\in F_\beta}v_\beta^\alpha\, e_\alpha.$$

By the $\Delta$-system lemma, we may additionally assume that there exists a finite subset $\Delta$ of $\Gamma$ such that $F_\beta\cap F_\gamma=\Delta$ for distinct $\beta,\gamma< \omega_1$. By linear independence of the $(v_\beta)_{\beta<\omega_1}$, we can also assume the sets $F_\beta\setminus\Delta$ to be non-empty; let us therefore choose $\alpha_\beta\in F_\beta\setminus\Delta$ ($\beta<\omega_1$). Finally, consider the functionals
$$\p_\beta:=\frac{1}{v_\beta^{\alpha_\beta}}e^*_{\alpha_\beta}$$
and observe that $\{v_\beta;\p_\beta\}_{\beta<\omega_1}$ is evidently a biorthogonal system in $Z_0$.
\end{proof}

\section{Large Banach spaces}\label{Sec: large}
In this last section we shall concentrate our attention on Banach spaces whose density character is at least equal to $\mathfrak{c}^+$ and we shall show that the results presented in the previous sections do not extend to this setting; in particular, we shall prove Theorem C (see Corollary \ref{Cor: large biorth}). We need one definition before explaining the results (see \cite[p. 254]{CN}).

\begin{definition} We say that a cardinal number $\kappa$ is \emph{strongly $\omega_1$-inaccessible}, and we write $\omega_1\Lt \kappa$, if $\alpha^\omega<\kappa$ whenever $\alpha<\kappa$.
\end{definition}

For example, we have $\omega_1\Lt\mathfrak{c}^+$ and, more generally, an infinite cardinal number $\kappa$ satisfies $\omega_1\Lt\kappa^+$ if and only if $\kappa^\omega=\kappa$. In particular, $\omega_1\Lt(2^\kappa)^+$, for every infinite cardinal number $\kappa$. Finally, let us mention that, subject to the Generalised Continuum Hypothesis, an infinite cardinal $\kappa$ satisfies $\omega_1\Lt\kappa^+$ if and only if ${\rm cf}\,\kappa>\omega$ (see, \emph{e.g.}, \cite[Theorems 3.11 and 5.15(iii)]{Jech}).

\begin{theorem}\label{Th: large disjoint support} Let $\Gamma$ be a cardinal number such that $\omega_1\Lt\Gamma$ and $X$ be a Banach space with ${\rm dens}\,X=\Gamma$ that admits an M-basis. Also let $Z$ be a subspace of $X$ with $|Z|\geq \Gamma$. Then, every maximal, disjointly supported family of unit vectors in $Z$ has cardinality $\Gamma$.

In particular, the result is true for every dense subspace $Z$ of $X$, or, more generally, for every subspace $Z$ such that ${\rm dens}\,Z=\Gamma$.
\end{theorem}
Evidently, Zorn's lemma then implies that there exist, in $Z$, disjointly supported families of unit vectors, of cardinality $\Gamma$.

\begin{proof} Fix an M-basis $\{e_\alpha;e^*_\alpha\}_{\alpha<\Gamma}$ for $X$ and assume, without loss of generality, that $\|e^*_\alpha\|=1$, for each $\alpha<\Gamma$. Also, let $\{z_\alpha\}_{\alpha<\gamma}$ be a collection of disjointly supported unit vectors in $Z$, with $|\gamma|<\Gamma$. We will prove that $\{z_\alpha\}_{\alpha<\gamma}$ is not maximal.\smallskip

Since, for each $\alpha<\gamma$, ${\rm supp}\,z_\alpha:=\{\beta<\Gamma\colon \langle e^*_\beta,z_\alpha \rangle\neq0\}$ is evidently a countable set, the set $S:=\cup_{\alpha<\gamma}{\rm supp}\,z_\alpha$ has cardinality $|S|\leq\max\{|\gamma|,\omega\}<\Gamma$. Consider then the bounded linear transformation $T\colon Z\to c_0(S)$ defined by
$$x\mapsto \{\langle e^*_\alpha,x\rangle\}_{\alpha\in S}.$$

On the one hand, we have $|c_0(S)|\leq |S|^\omega<\Gamma$, since $|S|<\Gamma$ and $\omega_1\Lt\Gamma$; on the other one, $|Z|\geq\Gamma$ by assumption. Consequently, the operator $T$ cannot be injective and we can select a unit vector $z_\gamma\in Z$ such that $Tz_\gamma=0$. In other words, the unit vectors $\{z_\alpha\}_{\alpha\leq\gamma}$ are disjointly supported, and we are done.
\end{proof}

\begin{remark} As it turns out, the condition $\omega_1\Lt\Gamma$ in Theorem \ref{Th: large disjoint support} is also necessary for the validity of the result. This can be seen from Corollary \ref{Cor: ortho large}, since it is proved in \cite{BCW} that the said corollary holds true precisely when  $\omega_1\Lt\Gamma$.
\end{remark}

Let us then turn to some consequences of the above result. For the sake of simplicity, we shall only state them for dense subspaces of a given Banach space $X$; however, as they only depend on the conclusion of the previous theorem, they are valid, more generally, for every subspace $Z$ such that $|Z|\geq{\rm dens}\,X$. Moreover, in each of them, the most interesting case is where ${\rm dens}\,X=\mathfrak{c}^+$.

\begin{corollary}\label{Cor: large biorth} Let $X$ be a Banach space with M-basis and such that $\omega_1\Lt{\rm dens}\,X$. Then every dense subspace $Z$ of $X$ contains a biorthogonal system of cardinality ${\rm dens}\,X$.
\end{corollary}
This result strongly contrasts with Theorem B and it clarifies the already mentioned fact that Theorem B does not extend to larger cardinalities. Moreover, the result offers a simple instance of a situation where an uncountable biorthogonal system can be constructed, even in absence of completeness. Let us also observe that an alternative, direct argument could be obtained from the same argument as in Lemma \ref{Lemma: Biort System in span of M-basis}, upon replacing the use of the $\Delta$-system lemma with its general version, \cite[Lemma III.6.15]{kunen}.

\begin{proof} Given a disjointly supported collection of unit vectors $(z_\alpha)_{\alpha<{\rm dens}\,X}$ in $Z$, select, for each $\alpha<{\rm dens}\,X$, an ordinal $\beta_\alpha\in{\rm supp}\,z_\alpha$. Then
$$\left\{z_\alpha; \frac{1}{\langle e^*_{\beta_\alpha},z_\alpha\rangle}e^*_{\beta_\alpha} \right\}_{\alpha<{\rm dens}\,X}$$
is the desired biorthogonal system.
\end{proof}

In the next result, we are even able to construct an Auerbach system of cardinality equal to the density of the underlying Banach space. This complements the results obtained in \cite{HKR} and, in a sense, also confirms the intuition in \cite[Problem 294]{GMZ} that some unconditionality assumption might imply the existence of large Auerbach systems.
\begin{corollary} Let $X$ be a Banach space with a long suppression $1$-unconditional Schauder basis and such that $\omega_1\Lt{\rm dens}\,X$. Then every dense subspace $Z$ of $X$ contains an Auerbach system of cardinality ${\rm dens}\,X$.
\end{corollary}

\begin{proof} Let us denote $\Gamma:={\rm dens}\,X$ and fix a long suppression $1$-unconditional Schauder basis $(e_\alpha)_{\alpha<\Gamma}$ for $X$; thus, for every subset $A$ of $\Gamma$ there is a naturally defined non-expansive projection $P_A\colon X\to \overline{\rm span}(e_\alpha)_{\alpha\in A}$. Fixed a dense subspace $Z$ of $X$, according to Theorem \ref{Th: large disjoint support}, we may select a disjointly supported family $(z_\alpha)_{\alpha<\Gamma}$ of unit vectors in $Z$.

For each $\alpha<\gamma$, we set $S_\alpha:={\rm supp}\,z_\alpha\subseteq\Gamma$ and we find a norming functional $z^*_\alpha$ for $z_\alpha$. Evidently, $\langle z^*_\alpha\circ P_{S_\alpha},z_\beta\rangle=0$ for distinct $\alpha,\beta\in \Gamma$ and $\|z^*_\alpha\circ P_{S_\alpha}\|\leq1$; therefore,
$$\left\{z_\alpha;z^*_\alpha\circ P_{S_\alpha}\right\}_{\alpha< \Gamma}$$
is the desired Auerbach system in $Z$.
\end{proof}

In the case where the Banach space $X$ is $\ell_p(\Gamma)$ ($1\leq p<\infty$), or $c_0(\Gamma)$ and we apply Theorem \ref{Th: large disjoint support} to the canonical basis $(e_\alpha)_{\alpha<\Gamma}$ of the space, it is evident that every disjointly supported family of unit vectors is isometrically equivalent to $(e_\alpha)_{\alpha<\Gamma}$. Therefore, we arrive at the next consequence of our result.
\begin{corollary} Let $\Gamma$ be a cardinal number such that $\omega_1\Lt\Gamma$ and let $X$ be either $\ell_p(\Gamma)$ ($1\leq p<\infty$), or $c_0(\Gamma)$. Then for every two dense subspaces $Y$ and $Z$ of $X$ there are subspaces $Y_0$ of $Y$ and $Z_0$ of $Z$ with ${\rm dens}\,Y_0={\rm dens}\,Z_0=\Gamma$ that are isometric.
\end{corollary}

This corollary implies, in particular, that the stronger version of Theorem A that we proved in Theorem \ref{th: WLD small} cannot be extended to WLD Banach spaces of density $\mathfrak{c}^+$; in other words, the dense subspaces that we built in Theorem \ref{Th: WLD large} cannot be made as diverse as the ones in Theorem \ref{th: WLD small}.\smallskip

To conclude, let us also mention that Theorem \ref{Th: large disjoint support} above subsumes the result from \cite{BCW} that we already mentioned in Section \ref{Sec: IPS}, concerning the existence of orthonormal systems in inner product spaces.
\begin{corollary}[\cite{BCW}]\label{Cor: ortho large} Let $\Gamma$ be a cardinal number such that $\omega_1\Lt\Gamma$. Then every dense subspace of $\ell_2(\Gamma)$ contains an orthonormal system of cardinality $\Gamma$.
\end{corollary}

{\bf Acknowledgements.} We are indebted to Tomasz Kochanek for pointing out to us the reference \cite{Gudder}. We also thank the anonymous referee for the helpful report.

\end{document}